\documentclass{gen-j-l}
\usepackage{amssymb}
\usepackage{amsfonts}

\newtheorem{theorem}{Theorem}[section]
\newtheorem{lemma}[theorem]{Lemma}
\theoremstyle{definition}

\newtheorem{example}[theorem]{Example}

\theoremstyle{remark}

\numberwithin{equation}{section}
\theoremstyle{plain}

\newtheorem{corollary}{Corollary}

\newtheorem{proposition}{Proposition}

\copyrightinfo{2001}{enter name of copyright holder}
\input{tcilatex}

\begin{document}
\title[\bigskip Condensed domains]{\bigskip Condensed domains and the $%
D+XL[X]$ construction}
\author{M. Zafrullah}
\address{Department of Mathematics, Idaho State University,\\
Pocatello, Idaho, USA}
\email{mzafrullah@usa.net}
\urladdr{http:/www.lohar.com}
\thanks{This paper is my nth swan song and may well be my last.}

\begin{abstract}
Let $D$ be an integral domain with quotient field $K$ and let $\mathcal{I}%
(D) $ be the set of nonzero ideals of $D$. Call, for $I,J\in \mathcal{I}(D)$%
, the product $IJ$ of ideals \emph{condensed} if $IJ=\{ij|i\in I,j\in J\}.$
Call $D$ a\emph{\ condensed domain} if for each pair $I,J$ the product $IJ$
is condensed. We show that if $a,b$ are elements of a condensed domain such
that $aD\cap bD=abD,$ then $(a,b)=D.$ It was shown in \cite{Zaf pres} that a
pre-Schreier domain is a $\ast $-domain, i.e., $D$ satisfies $\ast :$ For
every pair $\{a_{i}\}_{i=1}^{m},\{b_{j}\}_{j=1}^{n}$ of sets of nonzero
elements of $D$ we have $(\cap (a_{i}))(\cap b_{j})=\cap (a_{i}b_{j}).$ We
show that a condensed domain $D$ is pre-Schreier if and only if $D$ is a $%
\ast $-domain. We also show that if $A\subseteq B$ is an extension of
domains and $A+XB[X]$ is condensed, then $B$ must be a field and $A$ must be
condensed and in this case $[B:K]<4.$ In particular we study the necessary
and sufficient conditions for $D+XL[X]$ to be condensed, where $D$ is a
domain and $L$ an extension field of $K.$ It may be noted that if $D$ is not
a field $D[X]$ is never condensed. So for $D$ condensed $D+XK[X]$ is a way
of constructing new condensed domains from old.
\end{abstract}

\maketitle

\bigskip

Let $D$ be an integral domain with quotient field $K$ and let $\mathcal{I}%
(D) $ be the set of nonzero ideals of $D,$ throughout. Call, for $I,J\in 
\mathcal{I}(D),$ the product $IJ$ of ideals \emph{condensed} if $%
IJ=\{ij|i\in I,j\in J\}.$ We may call the ideals $I,J$ a \emph{condensed pair%
} if $IJ$ is condensed. Call $D$ a\emph{\ condensed domain} if for each $I,J$
the product $IJ$ is condensed. While we are at it, let's call an element $a$ 
\emph{subtle} if $a\in IJ$ implies that $a=ij$ where $i\in I$ and $j\in J.$
An element $a\in D$ is called irreducible or an atom if $a$ is a nonzero non
unit such that $a=xy$ implies $x$ is a unit or $y$ is, We show that if $D$
is condensed $a$ an atom and $b,c\in D$ with $(b,c)=D$, then $(a,b)=D$ or $%
(a,c)=D.$ We also show that if $a,b$ are elements of a condensed domain such
that $aD\cap bD=abD,$ then $(a,b)=D.$ Call $x\in D\backslash \{0\}$ primal
if for all $y,z\in D\backslash \{0\}$ $x|yz$ implies $x=rs$ where $r|y$ and $%
s|z.$ A domain all of whose nonzero elements are primal was called a \emph{%
pre-Schreier} domain in \cite{Zaf pres}. It was shown in \cite{Zaf pres}
that a pre-Schreier domain $D$ is a $\ast $-domain, i.e., $D$ satisfies $%
\ast :$ For every pair $\{a_{i}\}_{i=1}^{m},\{b_{j}\}_{j=1}^{n}$ of sets of
nonzero elements of $D$ we have $(\cap (a_{i}))(\cap b_{j})=\cap
(a_{i}b_{j}).$ We show that a condensed domain $D$ is pre-Schreier if and
only if $D$ is a $\ast $-domain. We also show that if $A\subseteq B$ is an
extension of domains and $A+XB[X]$ is condensed, then $B$ must be a field
and $A$ must be condensed and in this case $[B:K]<4.$ In particular we study
the necessary and sufficient conditions for $D+XL[X]$ to be condensed, where 
$D$ is a domain and $L$ an extension field of $K.$ It may be noted that if $%
D $ is not a field $D[X]$ is never condensed. So for $D$ condensed $D+XK[X]$
is a way of constructing new condensed domains from old.

Our basic tools come from the notion of star operations, as introduced in
sections 32 and 34 of Gilmer's \cite{G MIT}. For our purposes we provide
below a working introduction. Let $D$ be an integral domain with quotient
field $K$ and let $F(D)$ denote the set of fractional ideals of $D.$ Denote
by $A^{-1}$ the fractional ideal $D:_{K}A=\{x\in K|xA\subseteq D\}.$ The
function $A\mapsto A_{v}=(A^{-1})^{-1}$ on $F(D)$ is called the $v$%
-operation on $D$ (or on $F(D)).$ Associated to the $v$-operation is the $t$%
-operation on $F(D)$ defined by $A\mapsto A_{t}=\cup \{H_{v}|$ $H$ ranges
over nonzero finitely generated subideals of $A\}.$ The $v$- and $t$%
-operations are examples of the so called star operations. Indeed $%
A\subseteq A_{t}\subseteq A_{v}.$ A fractional ideal $A\in F(D)$ is called a 
$v$-ideal (resp., a $t$-ideal) if $A=A_{v}$ (resp., $A=A_{t})$ a $v$-ideal
(resp., a $t$-ideal) of finite type if there is a finitely generated ideal $%
B $ such $A=B_{v}$ (resp., $A=B_{t})$. An integral $t$-ideal maximal among
integral $t$-ideals is a prime ideal called a \emph{maximal }$t$-ideal. If $%
A $ is a nonzero integral ideal with $A_{t}\neq D$ then $A$ is contained in
at least one maximal $t$-ideal. A prime ideal that is also a $t$-ideal is
called a prime $t$-ideal. Every height one prime ideal is a $t$-ideal. Call $%
I\in F(D)$ $v$-invertible (resp., $t$-invertible) if $(II^{-1})_{v}=D$
(resp., $(II^{-1})_{t}=D).$ A prime $t$-ideal that is also $t$-invertible
was shown to be a maximal $t$-ideal in Proposition 1.3 of \cite[Theorem 1.4]%
{HZ t-inv}. Two elements $a,b\in D$ are said to be $v$-coprime if $%
(a,b)_{v}=D.$ Indeed $a,b$ are $v$-coprime if and only if $a,b$ share no
maximal $t$-ideals, if and only if $aD\cap bD=abD.$

\bigskip Let $X$ be an indeterminate over $K.$ Given a polynomial $g\in
K[X], $ let $A_{g}$ denote the fractional ideal of $D$ generated by the
coefficients of $g.$ A prime ideal $P$ of $D[X]$ is called a prime upper to $%
0$ if $P\cap D=(0).$ Thus a prime ideal $P$ of $D[X]$ is a prime upper to $0$
if and only if $P=h(X)K[X]\cap D[X],$ for a prime $h$ in $K[X].$ It follows
from \cite[Theorem 1.4]{HZ t-inv} that $P$ a prime upper to zero of $D$ is a
maximal $t$-ideal if and only if $P$ is $t$-invertible if and only if $P$
contains a polynomial $f$ such that $(A_{f})_{v}=D.$ A domain $D$ all of
whose prime uppers to zero are maximal $t$-ideals called a \emph{UMT domain}%
, \cite[Theorem 1.4]{HZ t-inv}. Our terminology is standard as in \cite{G
MIT} or is defined at the point of entry of the notion. We plan to split the
paper into two sections. In Section \ref{S1}, we collect basic properties of
condensed domains, some of which are known, some known with simpler proofs
and some new. Anderson and Dumitrescu in \cite{ADu} studied condensedness
for domains of the form $K+X^{r}L[[X]]$ where $K\subseteq L$ is an extension
of fields. In Section \ref{S2}, we study when a ring of the form $A+XB[X]$
is condensed and find the necessary and sufficient conditions for $D+XL[X]$
to be condensed, where $L$ is an extension field of $K.$ In particular we
show that $D$ is condensed if and only if $D+XK[X]$ is condensed.

\section{Basics\label{S1}}

An integrally closed pre-Schreier domain was originally called a Schreier
domain in \cite{C Bez} where it was indicated that the group of divisibility
of a Schreier domain is a Riesz group. Since the conclusion was based on the
fact that the nonzero elements of a Schreier domain are primal, one can
conclude that the group of divisibility of a pre-Schreier domain is a Riesz
group too. In an earlier version of \cite{Zaf pres} this author indicated
that one of the group theoretic characterizations of Riesz groups did not
translate to domains as a characterization of pre-Schreier domains. The
reason, in this author's opinion was the difference between the notions of
products of ideals in semigroups and in rings, see Section 2 of \cite{Zaf
pres}. This observation was related in an earlier version of \cite{Zaf pres}%
. Following the lead from that earlier version, D. F. Anderson and D. E.
Dobbs \cite{AD} introduced the concept of a condensed integral domain, as
defined in the introduction of this paper, see Corollary 2.6 of \cite{AD}.
They showed that $D$ is condensed if and only if every pair of two generated
ideals of $D$ are a condensed pair, if and only if every pair of finitely
generated ideals is a condensed pair, and that every overring of a condensed
domain is condensed. They also showed that a condensed domain $D$ has $%
Pic(D)=0$. Also, they showed that if a domain $D$ is not a field then $D[X]$
is not condensed and that if $F$ is a field $F[[X^{2}X^{3}]]$, is a
condensed domain. Later, Anderson, J. T. Arnold and Dobbs \cite{AAD} showed
that an integrally condensed domain is Bezout. A number of other researchers
have worked on concepts related to condensedness. An interested reader may
find \cite{ADu} a good source of information on this topic.

\bigskip

Lemma \label{Lemma A} Let $D$ be a condensed domain and let $a$ be an atom
in $D.$

\begin{enumerate}
\item If $b,c$ are co-maximal non-units of $D,$ then $a$ is co-maximal with $%
b$ or with $c.$

\item $a$ belongs to a unique maximal ideal of $D.$
\end{enumerate}

\begin{proof}
1. (1) Let $I=(a,b)$ and $J=(a,c).$ Then $IJ=(a^{2},ac,ab,bc)=(a,bc).$
Because $D$ is condensed, $a=(ra+sb)(ua+vc).$ So, $ua+vc$ is a unit or $%
ra+sb $ is.

2. (2) If $D$ is quasi local, then, clearly, $a$ belongs to a unique maximal
ideal. So let's assume that $D$ is non-local and that $a$ belongs to two
maximal ideals $M$ and $N.$ Let $m\in M\backslash N.$ So that $mD+N=D.$ That
is for some $n\in N$ we have $m,n$ co-maximal. By (1), $a$ is either
co-maximal with $m$ or with $n.$ But that is impossible because $a$ belongs
to both $M$ and $N.$ Whence $a$ belongs to a unique maximal ideal.
\end{proof}

As we shall see below $v$-coprime elements are co-maximal in a condensed
domain. For this we begin by recalling from \cite{DHLZ} some terminology. By
an overring of $D$ we mean a ring between $D$ and its quotient field $K.$
Let $D\subseteq R$ be an extension of domains. Then $R$ is said to be $t$%
-linked over $D$ if for each nonzero ideal $I$ of $D$ with $I^{-1}=D$ we
have $(IR)^{-1}=R$ and $D$ is $t$-linkative if every overring of $D$ is $t$%
-linked over $D.$

\begin{lemma}
\label{Lemma A1}Let $D$ be condensed and let $a,b$ be two nonzero non units
of $D.$ Then the following hold. (a) If $(a,b)_{v}=D$, then $(a,b)=D$. and
(b) If $I$ is a $t$-invertible ideal of $D,$ then $I$ is invertible and
hence principal.
\end{lemma}

\begin{proof}
(a) Every overring of a condensed domain is condensed by \cite{AD} and an
integrally closed condensed domain is Bezout by \cite{AAD}, as already
noted. So the integral closure of a condensed domain is Bezout, hence
Prufer. Thus $D$ is a $t$-linkative UMT domain, by Theorem 2.4 of \cite%
{DHLRZ} and every maximal ideal of $D$ is a $t$-ideal by Lemma 2.1 of \cite%
{DHLRZ}. Now let $(x,y)_{v}=D.$ Claim that $(x,y)=D.$ For if not then $(x,y)$
is contained in a maximal ideal $M$ of $D.$ But then $D=(x,y)_{v}\subseteq M$
a contradiction. For (b), let $II^{-1}\neq D.$ Then, being a proper integral
ideal, $II^{-1}$ is contained in a maximal ideal $M$. Now because the
integral closure of $D$ is Bezout every maximal ideal of $D$ is a maximal $t$%
-ideal as already noted. But then $II^{-1}\subseteq M$ gives a contradiction
as above. Whence $I$ is invertible. But an invertible ideal in a condensed
domain is principal, by Proposition 2.5 of \cite{AD}.
\end{proof}

There is another interesting application of the above observations. But let
us first record a simple fact which may be folklore, though I have not seen
it.

\begin{proposition}
\label{Proposition A2}An atom $a$ in a domain $D$ is a prime if and only if
for all $b\in D,$ $a\nmid b$ implies $(a,b)_{v}=D.$ Consequently, an atom in
a condensed domain is a prime if and only if for all $b\in D,a\nmid b$
implies $(a,b)=D.$
\end{proposition}

\begin{proof}
Suppose for all $b\in D$ $a\nmid b$ implies $(a,b)_{v}=D.$ Then for all $%
x,y\in D$ $a|xy$ implies $a$ divides $x$ or $a|y$. For if $a\nmid x,$ then $%
(a,x)_{v}=D$ by the condition. Yet as $a|xy$ we have $(a)=(a,xy).$ This
implies $%
(a)=(a,xy)_{v}=(a,ay,xy)_{v}=(a,(ay,xy))_{v}=(a,(a,x)y)_{v}=(a,(a,x)_{v}y)_{v}=(a,y)_{v}. 
$ Now $(a)=(a,y)_{v}$ implies $y\in (a)$ which is equivalent to $a|y.$
Conversely suppose that $a$ is a prime and $a\nmid b$ for some, chosen, $b$.
Then for each $h\in (a)\cap (b)$ we have $h=bt$ for $t\in D.$ Since $a\nmid
b $ we have $a|t.$ But then $t=at^{\prime }$ for some $t^{\prime }\in D$ and
so, for each $h\in (a)\cap \left( b\right) $ we have $h=abt^{\prime }$. But
this means $(a)\cap (b)=(ab)$ or $\frac{(a)\cap (b)}{ab}=D,$ or $%
(a,b)^{-1}=D $ which is equivalent to saying that $(a,b)_{v}=D.$ The
"consequently" part follows from the fact that in a condensed domain $%
(a,b)_{v}=D$ is equivalent to $(a,b)=D,$ by Lemma \ref{Lemma A1}.
\end{proof}

The above Proposition can be put to use immediately as follows.

\begin{corollary}
\label{Corollary A3} The following are equivalent for an atom $a$ in an
integral domain $D.$
\end{corollary}

(1) $a$ is a prime,

(2) $a$ generates a maximal $t$-ideal,

(3) if $a$ belongs to a prime ideal $P,$ then $a$ belongs to a maximal $t$%
-ideal contained in $P$

(4) if $a$ belongs to a prime $t$-ideal $P$ then $P$ is a maximal $t$-ideal
generated by $a.$

\begin{proof}
(1) $\Rightarrow $ (2). Let $\wp =(a)=\{ar|r\in D\}.$ Obviously, being a
principal ideal $(a)$ is a $t$-ideal. Let $M$ be a maximal $t$-ideal
containing $\wp $ and let $x\in M\backslash \wp $ and so $a\nmid x,$ by
construction. But by Proposition \ref{Proposition A2} $a\nmid x$ implies
that $(a,x)_{v}=D$ and this contradicts the assumption that $M$ is a $t$%
-ideal. Whence there is no $x\in M\backslash \wp $ and $\wp =M$ a maximal $t$%
-ideal.

(2) $\Rightarrow $ (3). Because $a\in P$ implies that $(a)\subseteq P$ and
by (2) $(a)$ is a maximal $t$-ideal.

(3) $\Rightarrow $ (4). Obvious because $a\in P$ implies that $(a)\subseteq
P $ and $(a)$ is a maximal $t$-ideal. Whence $(a)=P.$

(4) $\Rightarrow $ (1). Obvious because $a$ generates a prime.
\end{proof}

Note here that for an atom $a,$ $a\nmid b$ does not necessarily mean that $%
(a,b)_{v}=D.$ For example, let $D$ be a one-dimensional (Noetherian) local
domain and let $a,b$ be two non-associate atoms. Then $a\nmid b$ yet $%
(a,b)_{v}\neq D$ for the following two reasons. First: $a|b^{n}$ for some
positive integer $n,$ because $D$ is quasi-local and one dimensional and $%
(a,b)_{v}=D$ if and only if $(a,b^{n})_{v}$ for every positive integer $n$
(cf. \cite{Zaf v}). And second: $D$ is a one dimensional quasi-local domain
and so its maximal ideal is a $t$-ideal. For a concrete example note that if 
$F$ is a field and $X$ an indeterminate over $F,$ then $D=F[[X^{2},X^{3}]]$
is a one dimensional (Noetherian) local domain and of course $X^{2}$ and $%
X^{3}$ are two non associate atoms. (This domain is condensed, as already
noted.)

\begin{corollary}
\label{Corollary A4} In each of the following situations every prime element
generates a maximal ideal. (a) When every maximal ideal of $D$ is a $t$%
-ideal, i.e. when $D$ is $t$-linkative \cite{DHLZ}. (b) When $D$ has a
Prufer integral closure \cite{DHLRZ}.
\end{corollary}

\begin{proof}
Observe that (a) is obvious by Theorem 2.6 of \cite{DHLZ} and for (b) one
can recall from Theorem 2.4 of \cite{DHLRZ}, that $D$ has Prufer integral
closure if and only if $D$ is a $t$-linkative UMT domain.
\end{proof}

Now as we know that the integral closure of a condensed domain is Bezout we
have for the record the following corollary.

\begin{corollary}
\label{Corollary A5} In a condensed domain, every prime element generates a
maximal ideal and consequently $D[[X]]$ is a condensed domain if and only if 
$D$ is a field.
\end{corollary}

As already mentioned, Cohn \cite{C Bez} called an integrally closed integral
domain $D$ Schreier if each nonzero element of $D$ is primal. A domain whose
nonzero elements are primal was called pre-Schreier in \cite{Zaf pres}. Note
that in a pre-Schreier domain every irreducible element (atom) is a prime.
(In fact a primal atom in any domain, is prime. For let $p$ be an
irreducible element that is also primal and let $p|ab.$ So $p=rs$ where $r|a$
and $s|b$, because $p$ is primal. But as $p$ is also an atom, $r$ is a unit
or $s$ is a unit. Whence $p|a$ or $p|b$. In studying pre-Schreier domains, I
came across a property that I called the property $\ast $. It was defined in
the introduction.

It was shown in Theorem 1.6 of \cite{Zaf pres} that $D$ is a pre-Schreier
domain if and only if for each pair $\{a_{i}\}_{i=1}^{m},\{b_{j}\}_{j=1}^{n}$
of sets of nonzero elements of $D$ and for all $x\in D\backslash \{0\}$ $%
a_{i}b_{j}|x$ implies $x=rs$ where $a_{i}|r$ and $b_{j}|s,$ $i=1...m$ and $%
j=1...n.$ This result can be used to prove the following proposition.

\begin{proposition}
\label{Proposition A6}A domain $D$ is a pre-Schreier domain if and only if $%
D $ is a $\ast $-domain such that for every pair $\{a_{i}\}_{i=1}^{m},%
\{b_{j}\}_{j=1}^{n}$ of sets of nonzero elements of $D$ $(\cap
(a_{i})),(\cap (b_{j}))$ is a condensed pair.
\end{proposition}

\begin{proof}
Let $D$ be a pre-Schreier domain. That $D$ is a $\ast $-domain follows from
(1) of \cite[Corollary 1.7]{Zaf pres}. Now let $\{a_{i}\}_{i=1}^{m},\{b_{j}%
\}_{j=1}^{n}$ be a pair of sets of nonzero elements of $D$ such that $%
a_{i}b_{j}|x.$ Then $x=rs$ where $a_{i}|r$ and $b_{j}|s.$ Now as $%
a_{i}b_{j}|x$ if and only if $x\in \cap (a_{i}b_{j})$ and $a_{i}|r$ and $%
b_{j}|s$ if and only if $r\in \cap (a_{i})$ and $s\in \cap (b_{j}).$ Thus by
the pre-Schreier property $x\in \cap (a_{i}b_{j}$ implies that $x=rs$ where $%
r\in \cap (a_{i})$ and $s\in \cap (b_{j}).$ But as we already have
established that $D$ has the $\ast $-property, $\cap (a_{i}b_{j})=(\cap
(a_{i}))(\cap (b_{j})).$ Thus $x\in (\cap (a_{i}))(\cap (b_{j}))$ implies
that $x=rs$ where $r\in \cap (a_{i})$ and $s\in \cap (b_{j})$ and $(\cap
(a_{i})),(\cap (b_{j}))$ is a condensed pair. For the converse suppose that $%
D$ is a $\ast $-domain and for $\{a_{i}\}_{i=1}^{m},\{b_{j}\}_{j=1}^{n}%
\subseteq D,$ $(\cap (a_{i})),(\cap (b_{j}))$ is a condensed pair. Because $%
(\cap (a_{i})),(\cap (b_{j}))$ is a condensed pair, $x\in (\cap
(a_{i}))(\cap (b_{j}))$ implies that $x=rs$ where $r\in (\cap (a_{i}))$ and $%
s\in (\cap (b_{j})).$ But since $D$ has the $\ast $-property, $(\cap
(a_{i}))(\cap (b_{j}))=\cap (a_{i}b_{j}).$ Thus $x\in \cap (a_{i}b_{j})$
implies that $x=rs$ where $r\in (\cap (a_{i}))$ and $s\in (\cap (b_{j})).$
which translates to $a_{i}b_{j}|x$ implies $x=rs$ where $a_{i}|r$ and $%
b_{j}|s$ and according to Theorem 1.6 of \cite{Zaf pres} this is the
characterizing property of pre-Schreier domains.
\end{proof}

The above Proposition can be used to prove the following result.

\begin{proposition}
\label{Proposition A7} If $D$ is condensed and a $\ast $-domain, then $D$ is
a pre-Schreier domain.
\end{proposition}

\begin{proof}
If $D$ is condensed, then every pair of nonzero ideals of $D$ is condensed
and so is $(\cap (a_{i})),(\cap (b_{j})),$ for any pair $\{a_{i}%
\}_{i=1}^{m},\{b_{j}\}_{j=1}^{n}$ of sets of nonzero elements of $D.$ But
then, being a $\ast $-domain makes $D$ a pre- Schreier domain.
\end{proof}

Now these simple observations have the following somewhat interesting
implications.

\begin{corollary}
\label{Corollary A8} An atomic condensed domain $D$ is a PID if and only if $%
D$ has the $\ast $-property. Consequently a non-integrally closed atomic
condensed domain $D$ does not have the $\ast $-property.
\end{corollary}

\begin{proof}
Let $D$ be atomic and condensed and suppose that $D$ has the $\ast $%
-property. Then, by Proposition \ref{Proposition A7}, $D$ is pre-Schreier.
But every atom is a prime in a pre-Schreier domain. So, being an atomic
domain, $D$ is a UFD. But then $D$ is integrally closed and an integrally
closed condensed domain is Bezout, \cite{AAD}. Whence $D$ is a PID. Of
course a PID has the $\ast $-property and is condensed.
\end{proof}

\begin{example}
\label{Example A9} Let $K$ be a field, let $X$ be an indeterminate over $K$
and let $D=K[[X^{2},X^{3}]].$ Then $D$ does not satisfy $\ast $. The reasons
are (a) $D$ is Noetherian, (b) according to \cite{AD} $D$ is condensed and
(c) $D$ is not integrally closed.
\end{example}

Now recall the "number crunching" I had to do in Example 2.8 of \cite{Zaf
pres} to establish that $K[[X^{2},X^{3}]]$ was not a $\ast $-domain. (Of
course the above approach offers a simpler and direct route compared to the
alternate suggested in \cite{Zaf pres}.) There may arise a question here: Is
a pre-Schreier domain condensed? The answer is: generally, it is not the
case. For example if $D$ is a Schreier domain then it is well known that $%
D[X]$ is Schreier (cf. \cite{C Bez}) and Schreier is integrally closed
pre-Schreier. Now if $D$ is not a field then, as we have noted above (see
Proposition \ref{Proposition B} below as well), $D[X]$ can never be a
condensed domain.

Usually, $D$ having the $\ast $-property does not mean that $D$ is
integrally closed and this is established by the existence of a pre-Schreier
domain that is not Schreier \cite{Zaf pres}, yet there are situations where
the presence of the $\ast $-property in $D$ ensures that $D$ is ("more than"
integrally closed. Call an integral domain $D$ $v$-coherent if for each
nonzero finitely generated ideal $I$ of $D$ we have $I^{-1}$ a $v$-ideal of
finite type. Also call $D$ a generalized GCD (GGCD) domain if for each pair
of nonzero elements of $D$ we have $aD\cap bD$ invertible. It is well known
that a GGCD domain is a locally GCD domain, i.e. $D_{M}$ is a GCD domain for
each maximal ideal $M$, and hence is integrally closed, \cite{AAGGCD}. So a
condensed GGCD domain being Bezout is as given as a a Prufer domain becoming
Bezout being condensed. However the following may well be an improvement on
Corollary 2.6 of \cite{AD}. For this recall that $D$ is a $v$-finite
conductor domain if for every pair of nonzero elements $a,b$ of $D$ the
ideal $aD\cap bD$ is a $v$-ideal of finite type.

\begin{corollary}
\label{Corollary A10} Let $D$ be a condensed domain that is also a $v$%
-finite conductor domain. Then $D$ is a Bezout domain if and only if $D$ is
a $\ast $-domain.
\end{corollary}

That a Bezout domain is a $\ast $-domain follows from the fact that every
GCD domain is Schreier \cite{C Bez} and hence a $\ast $-domain \cite[Theorem
3.6]{Zaf pres}. For the converse note that by Proposition \ref{Proposition
A7} a condensed $\ast $-domain is pre-Schreier and a pre-Schreier $v$-finite
conductor domain is a GCD domain \cite[Theorem 3.6]{Zaf pres} and a GCD
domain is integrally closed.

It may be noted, however, that a condensed $v$-finite conductor domain, even
a condensed Noetherian domain may not be Bezout, as the example of $%
K[[X^{2},X^{3}]]$ indicates. If you go chasing the facts they will take you
further a field, with negative results as it were. Here's a slightly
advanced form of Noetherian domains, recently introduced by this author in 
\cite{Zaf g-ded}. Call $D$ a \emph{dually compact domain} (\emph{DCD}) if
for each set $\{a_{\alpha }\}_{\alpha \in I}$ $\subseteq K\backslash \{0\}$
with $\cap a_{\alpha }D\neq (0)$ there is a finite set of elements $%
\{x_{1},...,x_{r}\}\subseteq K\backslash \{0\}$ such that $\cap a_{\alpha
}D=\cap _{i=1}^{r}x_{i}D$, or equivalently for each $I\in F(D)$, the ideal $%
I_{v}=(I^{-1})^{-1}$ is a finite intersection of principal fractional ideals
of $D$. Indeed a DCD can be condensed without being Bezout. The reason is
that a DC domain will become $v$-G-Dedekind, only if it is a $\ast $-domain,
as shown in Theorem 3.3 of \cite{Zaf g-ded}. Here a domain $D$ is a $v$%
-G-Dedekind domain if $I_{v}$ is invertible for each $I\in F(D).$ But as
soon as you add the condensed property, you get a Bezout domain, because a $%
v $-G-Dedekind domain is integrally closed. On the other hand make a DC
domain as condensed as you want, it won't become Bezout unless it is a $\ast 
$-domain.

\section{New condensed domains from old \label{S2}}

The following is a known result (see e.g. \cite{AD}), but our proof may be
very simple.

\begin{proposition}
\label{Proposition B}Let $D$ be an integral domain and $X$ an indeterminate
over $D.$ Then $D[X]$ is condensed if and only if $D$ is a field.
Consequently if $D$ is a domain such that $D$ is not a field, then $D[X]$ is
never condensed.
\end{proposition}

\begin{proof}
Certainly $X$ is irreducible and hence, by (2) of Lemma A, must belong to a
unique maximal ideal of $D[X].$ But that is possible only if $D$ is a field.
(Alternatively note that $X$ is a prime in $D[X]$ and if $D[X]$ is
condensed, then $X$ must generate a maximal ideal which is possible only if $%
D$ is a field.) Conversely if $D$ is a field, then $D[X]$ is PID and hence,
obviously, a condensed domain. The consequently part is obvious.
\end{proof}

\begin{proposition}
\label{Proposition C} Let $A\subseteq B$ be an extension of domains such
that $(A:B)\neq (0).$ If $A$ is condensed, then so is $B.$
\end{proposition}

\begin{proof}
Let $I,J\in \mathcal{I}(B).$ Then for some $\alpha ,\beta \in \lbrack A:B]$
we have $\alpha \beta IJ=(\alpha I)(\beta J),$ where $(\alpha I),(\beta
J)\in \mathcal{I}(A).$ So for $x\in IJ,$ we have $\alpha \beta x\in \alpha
\beta IJ=(\alpha I)(\beta J),$ forcing $\alpha \beta x=rs$ where $r\in
\alpha I$ and $s\in \beta J,$ because $A$ is condensed. This gives $r/\alpha
\in I$ and $s/\beta \in J.$ But as $\alpha \beta x=rs,$ we have $x=(r/\alpha
)(s/\beta ).$
\end{proof}

\begin{proposition}
\label{Proposition D} Let $A\subseteq B$ be an extension of domains. If $%
A+XB[X]$ is a condensed domain, then $B$ is a field and $A$ is a condensed
domain.
\end{proposition}

\begin{proof}
Since $(A+XB[X]:B[X])=XB[X],$ we conclude from Proposition \ref{Proposition
C} that $B[X]$ is condensed. But by Proposition \ref{Proposition B}, $B$
must be a field. Next, let $I,J\in \mathcal{I}(A)$ and let $a\in IJ.$ Since $%
B$ is a field $I(A+XB[X])=I+XB[X]$, $J(A+XB[X])=J+XB[X]$ and $%
(I+XB[X])(J+XB[X])=IJ+XB[X].$ Now $a\in IJ\backslash \{0\}$ means $a\in $ $%
IJ+XB[X]=(I+XB[X])(J+XB[X]).$ This means $a=f_{1}f_{2}$ where $f_{1}\in
(I+XB[X])$ and $f_{2}\in (J+XB[X]),$ because $A+XB[X]$ is a condensed
domain. Now $f_{1}=r+Xg_{1}(X)$ and $f_{2}=s+g_{2}(X)$ where $r\in I$ and $%
s\in J.$ Thus $a=(r+Xg_{1}(X))($ $%
s+Xg_{2}(X))=(rs+X(rg_{2}(X)+sg_{1}(X))+X^{2}g_{1}(X)g_{2}(X)).$ Comparing
coefficients, $a=rs$ where $r\in I$ and $s\in J.$
\end{proof}

Note that $(A+XB[[X]]:B[[X]])=XB[[X]]$ and ideals of $A+XB[[X]]$ are of the
form $I+XB[[X]]$ where $I$ is an ideal of $D$ or of the form $X^{r}JXL[[X]]$
where $J$ is a $D$-submodule of $L$ (see e.g. Proposition 2.6 of \cite{ADu}%
). With reference to Corollary \ref{Corollary A5} we have the following
Corollary.

\begin{corollary}
\label{Corollary D1} Let $A\subseteq B$ be an extension of domains. If $%
A+XB[[X]]$ is a condensed domain, then $B$ is a field and $A$ is a condensed
domain.
\end{corollary}

For the converse of Proposition \ref{Proposition D} we need to digress a
little and recall Proposition 3 of \cite{Zaf potent}.

\begin{proposition}
\label{Proposition E} Let $D$ be an integral domain and let $L$ be an
extension field of the field of fractions $K$ of $D.$ Then each nonzero
ideal $F$ of $R=D+XL[X]$ is of the form $f(X)JR=f(X)(J+XL[X])$ , where $J$
is a $D$-submodule of $L$ and $f(X)\in R$ such that $f(0)J\subseteq D.$ If $%
F $ is finitely generated, $J$ is a finitely generated $D$-submodule of $L.$
\end{proposition}

Using the tail-end part of the proof of the above proposition, we can
conclude that if $F$ is a two generated ideal of $R$, then $F=f(X)(J+XL[X])$
where $J$ is a two generated $D$-submodule of $L$ and $f(X)\in R.$ The
following special cases apply:

\begin{itemize}
\item (a) If $f(X)=1,$ $J$ is an ideal of $D$ and

\item (b) If $f(X)$ is non constant with $f(0)\neq (0),$ $J$ is still a
fractional ideal of $D$. By replacing $f$ by $\frac{1}{d}f$ we can assume
that $J$ is an ideal of $D$ (as in that case $f(0)=1$). Because $f(0)=1$ we
have $f(X)\in R$ and $J$ is an ideal (since $f(0)J\subseteq D)$ and so the
case (b) reduces to case (a). This leaves the case of

\item (c) for $f(0)=0$. If $f(0)=0,$ then $f(X)=X^{r}g(X)$ where $r>0$ and $%
g(0)=1.$ (We can assume that because if $g(0)=l\in L\backslash \{0\},$ we
can replace the generators $j_{i}$ of $J$ by $j_{i}/l).$ Now suppose that $D$
is condensed and we want to show that $R=D+XL[X]$ is condensed. By Theorem 1
of \cite{AD}, we need to show that the product of any pair $A,B$ of nonzero $%
2$-generated (or finitely generated) ideals of $R$ is condensed. But the
general $D+XL[X]$ case may be hard, as indicated in \cite{ADu}. So, let's
take care of the simpler cases before attacking the harder one(s). The first
of the simpler cases is tackled in the following Lemma.
\end{itemize}

\begin{lemma}
\label{Lemma E1} If $A=X^{r}g(X)L[X]$, where $g(0)=1$ then the pair $A,B$ is
condensed for any ideal $B$ of $R=D+XL[X]$.
\end{lemma}

Proof. Indeed if $Aa,$ $B$ is a condensed pair, where $a\in D\backslash 0\},$
then so is $A,B.$ This is because if $x\in (Aa)B$ implies $x=rs$ where $r\in
Aa$ and $s\in B,$ then $y\in AB$ implies $ya\in $ $(Aa)B,$ forcing $ya=r\in
Aa$ and $s\in B$ and thus $y=(r/a)s.$ That $A,B$ being condensed implying $%
Aa,B$ being condensed is direct. Consequently we can take $A=XL[X].$ The
other ideal could be (a) $A_{1}=XL[X]$ or (b) $B=\mathfrak{B}+XL[X],$ where $%
\mathfrak{B}$ is a nonzero ideal of $D$ or (c) $C=X(\mathfrak{C}+XL[X],$
where $\mathfrak{C}$ is a nonzero $D$-submodule of $K.$ In case (a,a) we
have $AA_{1}=X^{2}L[X]$ and $x\in AA_{1}$ implies $x=X^{2}h(X)$ and we can
set $x=(X)(Xh(X)).$ For the case (a,b) we have $AB=(XL[X])(\mathfrak{B}%
+XL[X])=XL[X]$ and $x\in AB$ implies $x=Xh(X)$ where $h(X)\in K[X]$ and we
can find $d\in D\backslash \{0\}$ such that $dh(X)\in R.$ In this case $%
x=(X/d)(dXh(X))$ will do, as $X/d\in XL[X]$ and $dXh(X)\in \mathfrak{B}%
+XL[X].$ Finally, in case (a,c) we have $AC=(XL[X])(X(\mathfrak{C}%
+XL[X])=X^{2}L[X]$ and $x\in AC$ means $x=X^{2}h(X)$ where $h(X)\in L[X].$
We can find $l\in L\backslash \{0\}$ so that $lh(X)\in (\mathfrak{C}+XL[X]$
and set $x=(X/l)(lh(X).$

Alternatively, let $x\in (XL[X])B.$ Then $x=\sum Xf_{i}b_{i}.$ Since $%
f_{i}\in L[X]$ we can find $l\in L\backslash \{0\}$ such that $lf_{i}\in R.$
But then $x=(X/l)(\sum lf_{i}b_{i}).$ Now as $lf_{i}\in R$ we have $\sum
lf_{i}b_{i}\in B.$ But then we have an expression for $x$ in the required
form.

As an application of Lemma \ref{Lemma E1} when considering condensedness of
two nonzero ideals $I,J$ of $R,$ we can avoid the cases where one of the
ideals if of the form $A=X^{r}g(X)L[X].$ The following result can be proved
as a corollary of a latter result, but we prove it separately for the sake
of clarity.

\begin{theorem}
\label{Theorem F}Let $D$ be a domain, $K$ the quotient field of $D$ and let $%
X$ be an indeterminate over $K.$ Then $D$ is condensed if and only if $%
D+XK[X]$ is condensed.
\end{theorem}

For a start let us display below the types of ideals that we may expect in
our study, with reference to Proposition 4.12 of \cite{CMZ}.

\begin{itemize}
\item (a) When $f(X)=1$ we have $A=(\mathfrak{A}+XK[X])$ where $\mathfrak{A}$
is a $2$ generated ideal of $D$ and $\mathfrak{A}\neq (0),$ by Lemma \ref%
{Lemma E1}

\item (b) When $f(X)$ is such that $f(0)=1$ we have $B=f(X)(\mathfrak{B}%
+XK[X]),$ where $\mathfrak{B}$ is a nonzero two generated ideal of $D,%
\mathfrak{B}\neq (0),$ by Lemma \ref{Lemma E1}. Since $f$ belongs to $R,$
case (b) reduces to case (a).

\item (c) When $f(X)=X^{r}g(X)$, with $g(0)=1,$ where $r$ is a positive
integer and $\mathfrak{C}$ is a nonzero $2$-generated fractional ideal of $D$%
. But as $X^{r-1}g\in R$ we get $C=X(\mathfrak{C}+XK[X])$
\end{itemize}

Depending on the types of the $2$-generated ideals we need to study the
following three cases (a,a), (a,c) ( c,c).

\begin{itemize}
\item (aa) $A=(\mathfrak{A}+XK[X]),$ $A_{1}=(\mathfrak{A}_{1}+XK[X]),$ $%
AA_{1}=(\mathfrak{AA}_{1}+XK[X].$ Now $x\in AA_{1}$ implies $x=ij+Xh(X),$
where $i\in $ $\mathfrak{A}$ and $j\in \mathfrak{A}_{1}$ because the product 
$\mathfrak{AA}_{1}$ is condensed. We can write $a=i\in \mathfrak{A}+XK[X]$,
where $i\in \mathfrak{A}$ and $b=(j+(X/i)h(X))\in \mathfrak{A}_{1}+XK[X],$
here $Xh(X)\in R$ and so $(X/i)h(X)\in XK[X].$

\item (ac) $A=(\mathfrak{A}+XK[X]),~C=X(\mathfrak{C}+XK[X]).$ Then $AC=X(%
\mathfrak{AC}+XK[X])$. Let $x\in AC.$ Then $x=X(\gamma +Xh(X))$ where $%
\gamma \in \mathfrak{AC.}$ Since $\mathfrak{A}$ is an ideal and $\mathfrak{%
C=I/d}$ where $\mathfrak{d}$ is a nonzero element of $D,$ $\mathfrak{A,C}$
is a condensed pair and we can write $\gamma =\alpha \beta $ where $\alpha
\in \mathfrak{A}$ and $\beta \in \mathfrak{C=I/d.}$ Set $a=\alpha $ and $%
c=(X(\beta +(X/\alpha )h(X)).$ Since $(X/\alpha )h(X)\in XK[X]$ we have $%
\beta +(X/\alpha )h(X)\in \mathfrak{C}+XK[X]$ and so $X(\beta +(X/\alpha
)h(X))\in X(\mathfrak{C}+XK[X]).$ Thus $x=ac=\alpha (X(\beta +(X/\alpha
)h(X))).$

\item (cc) $C_{1}=X(\mathfrak{C}_{1}+XK[X]),~C_{2}=X(\mathfrak{C}%
_{2}+XL[X]), $ $C_{1}C_{2}=X^{2}(\mathfrak{C}_{1}\mathfrak{C}_{2}+XK[X]).$
Let $x\in C_{1}C_{2}$ and let $\gamma \in \mathfrak{C1C}_{2}\backslash \{0\}.
$ Then $x=X^{2}(\mathfrak{\gamma }+Xh(X)).$ Here too we must find $\gamma
_{1}\in \mathfrak{C}_{1}$ and $\gamma _{2}\in \mathfrak{C}_{2}$ such that $%
\gamma =\gamma _{1}\gamma _{2}.$ But this is easy in this case because by
Proposition 4.12 of \cite{CMZ}, and our assumption that $g(0)=1,$ $\mathfrak{%
C}_{1}$ and $\mathfrak{C}_{2}$ are both fractional ideals of $D.$ So $%
\mathfrak{C}_{i}=\frac{I_{i}}{d_{i}}$ where $I_{i}$ are ideals of $D$ and $%
d_{i}\in D\backslash \{0\}.$ Thus $\gamma =\frac{y}{d_{1}d_{2}}$ and as $D$
is condensed, $y=y_{1}y_{2}$ where $y_{i}\in I_{i}$ and so $\gamma =(\gamma
_{1})(\gamma _{2})$ where $\gamma _{i}=\frac{y_{i}}{d_{i}}\in \mathfrak{C}%
_{i}.$ Set $c_{1}=X\gamma _{1}\in C_{1}$ and set $c_{2}=X(\gamma _{2}+\frac{X%
}{\gamma _{1}}h(X)).$ Now $X\gamma _{2}\in C_{2}$ patently because $\gamma
_{2}\in \mathfrak{C}_{2}$ and $X(\frac{X}{\gamma _{1}}h(X))\in C_{2}$
because $X(\frac{X}{\gamma _{1}}h(X))\in X(XL[X]).$ Since both belong to $X(%
\mathfrak{C}_{2}+XK[X)$ their sum must do the same. Now check that $%
c_{1}c_{2}=$ $c=(X^{r}g_{1}(X)\gamma _{1})($ $X^{s}g_{2}(X)(\gamma _{2}+%
\frac{X}{\gamma _{1}}h(X))$. That $D$ is condensed if $D+XK[X]$ is condensed
follows from Proposition \ref{Proposition D}.
\end{itemize}

Another simple case is that of when $D$ is a field, though here we shall
consider the ring $K+XL[X]$ where $L$ is an extension of $K$. Let us first
write another version of Proposition 3 of \cite{Zaf potent}: Let $K$ be a
field, $L$ an extension field of $K$ and let $X$ be an indeterminate over $%
L. $ Then each nonzero ideal $F$ of $R=K+XL[X]$ is of the form $%
F=f(X)JR=f(X)(J+XL[X])$ ,where $J$ is a $K$-subspace of $L$ and $f(X)\in R$
such that $f(0)J\subseteq K$. If $F$ is finitely generated, $J$ is a
finitely generated $K$-subspace of $L$.

Now in this case $f(X)=1$, as $f\in R,$ gives $F$ as either $F=K+XL[X]=R$
(if $F\cap K\neq (0)$) or $F=XL[X]$ (if $F\cap D=(0)$). (We could have had $%
F=f(X)X^{r}L[X],$ but the considerations like the ones in the proof of Lemma %
\ref{Lemma E1} would have whittled it down to the current form.) Next for $f$
such that $f(0)=1$ we have $F=f(X)R$ (when $J\neq 0$) and $F=f(X)XL[X]$
(when $J=(0)$). In the $f(0)=0$ case we have $F=X^{r}g(X)JR$ where $J$ is a $%
K$-submodule of $L.$

Of these $f(X)R$, being principal, will produce a condensed pair with any
ideal $J$ of $R.$ So will $f(X)XL[X]$.

So, essentially, we have two types of ideals that need to be considered (a) $%
A=X^{r}g(X)XL[X]$ (or $A=XL[X]$ as $X^{r}g(X)\in R$ and so can be
cancelled.) and (b) $B=F=X^{s}g(X)JR$ where $J$ is a $K$-submodule of $L.$

\begin{lemma}
\label{Lemma G} $XL[X]$ , $A$ is a condensed pair for every ideal $A$ of $R.$
\end{lemma}

The proof works as in Lemma \ref{Lemma E1}. Let $x\in XL[X]A.$ Then $x=\sum
Xf_{i}(X)a_{i}(X).$ Since $f_{i}\in L[X]$ we have $(f_{1},...,f_{n})=f$,
because $L[X]$ is a PID. So $f_{i}=h_{i}(X)f(X)$ and $%
Xf_{i}=l_{i}Xh_{i}(X)f(X)$ where $h_{i}(X)\in R.$ But then $\sum
l_{i}Xf=lXf\in XL[X]$ (because $L$ is a field) and $\sum h_{i}a_{i}(X)\in A,$
because $h_{i}\in R$. Thus $x=g(X)a(X)$ where $g\in XL[X]$ and $a(X)\in A.$

Next note that, in $F=X^{s}g(X)JR$, or in $F=XJR,$ $F$ is $2$-generated if
and only if $J$ is a $2$-generated $K$-subspace of $L$. of $r.$

Let $K\subseteq L$ be an extension of fields. In the second section of \cite%
{ADu} Anderson and Dumitrescu introduce the notion of $K\subseteq L$ being $%
vs$-closed as follows. Let $V,W$ be two $K$-subspaces of $L.$ Let $%
P(V,W)=\{vw|v\in V$ and $w\in W\}$ and let $VW$ denote the $K$-subspace of $%
L $ generated by $P(V,W).$ Call $K\subseteq L$ $vs$-closed if for each pair $%
V,W$ of $K$-subspaces of $L$ we have $VW=P(V,W).$ According to Proposition
2.6 of \cite{ADu}, $K\subseteq L$ is $vs$-closed if and only if for every $%
\alpha ,\beta \in L$, $1$+$\alpha \beta $ $=(a+b\alpha )(c+d\beta )$ for
some $a,b,c,d\in $ $K.$ Using the fact that if $[L:K]\geq 4$ then $L$
affords a pair of elements $\alpha ,\beta $ such that $1,\alpha ,\beta
,\alpha \beta $ are linearly independent over $K$ the authors of \cite{ADu}
concluded that when $K\subseteq L$ is $vs$-closed $[L:K]\leq 3.$

\begin{lemma}
\label{Lemma H} The ring $R=K+XL[X]$ is condensed if and only if for every
pair of distinct ideals of the form $C=X(J+XL[X])$ where $J$ is a strictly
two generated nonzero $K$-subspace of $L,$ is condensed.
\end{lemma}

\begin{proof}
Indeed the assertion holds if $R$ is condensed. For the converse we note,
using the observations prior to Lemma \ref{Lemma G}, that $R$ has proper
ideals of the following types: (a) $A=XL[X]$, and this covers the case of $%
f(X)XL[X]$, where $J=0$ and $f(0)=1.$ (Because if $f\in R\backslash \{0\},$ $%
(fA,B)$ is a condensed pair if and only if $(A,B)$ is a condensed pair), (b) 
$B=f(X)R,$ but this is principal and will form a condensed pair with every
other ideal, as we have already observed.) This leaves ideals of the type
(c) $C=X^{r}g(X)(J+XL[X])$ where $J$ is a strictly two generated nonzero $K$
subspace of $L.$ Now with reference to the proof of Theorem \ref{Theorem F}
the cases of (a,a), (a,b) and (a,c) have been settled in Lemma \ref{Lemma G}%
. The cases of (b,b) and (b,c) are settled because $B$ is nonzero principal.
That leaves the case of (c,c) and that establishes the lemma.
\end{proof}

\begin{proposition}
\label{Proposition J} Let $K\subseteq L$ be an extension of fields, let $X$
be an indeterminate over $L$ and let $R=K+XL[X].$ Then $R$ is condensed if
and only if $K\subseteq L$ is $vs$-closed. Moreover if $[L:K]\geq 4,$ $%
R=K+XL[X]$ is not condensed.
\end{proposition}

\begin{proof}
Suppose that $K\subseteq L$ is $vs$-closed. By Lemma \ref{Lemma H}, all we
have to do is study the case (c,c) of pairs of two generated ideals of the
form $C=X(\mathfrak{C}+XL[X]).$ That is $C_{1}=X(\mathfrak{C}_{1}+XL[X])$
and $C_{2}=X(\mathfrak{C}_{2}+XL[X]).$ $C_{1}C_{2}=X^{2}(\mathfrak{C}_{1}%
\mathfrak{C}_{2}+XL[X]).$ Let $x\in C_{1}C_{2}$ and let $\gamma \in 
\mathfrak{C1C}_{2}\{0\}.$ Then $x=X^{2}(\mathfrak{\gamma }+Xh(X)).$ But $%
\mathfrak{C}_{1}$ and $\mathfrak{C}_{2}$ are both $K$-subspaces of $L.$ So $%
\mathfrak{C}_{i}=(l_{i1},l_{i2})$ where $l_{ij}$ are elements of $L.$ As $%
K\subseteq L$ is $vs$-closed $\gamma =(\gamma _{1})(\gamma _{2})$ where $%
\gamma _{i}=(k_{i1}l_{i1}+k_{i1}l_{i2})\in \mathfrak{C}_{i}.$ Set $%
c_{1}=X\gamma _{1}\in C_{1}$ and set $c_{2}=X(\gamma _{2}+\frac{X}{\gamma
_{1}}h(X)).$ Now $X\gamma _{2}\in C_{2}$ patently because $\gamma _{2}\in 
\mathfrak{C}_{2}$ and $X(\frac{X}{\gamma _{1}}h(X))\in C_{2}$ because $X(%
\frac{X}{\gamma _{1}}h(X))\in X(XL[X]).$ Since both belong to $X^{s}g_{2}(X)(%
\mathfrak{C}_{2}+XL[X)$ their sum must do the same. Now check that $%
c_{1}c_{2}=$ $c=(X\gamma _{1})$ $X(\gamma _{2}+\frac{X}{\gamma _{1}}h(X))$.
The converse can be proved as follows. Suppose that $K+XL[X]$ is condensed,
then for each pair of two generated nonzero ideals $C_{1}=X^{r}g(X)(%
\mathfrak{C}_{1}+XL[X])$ and $C_{2}=X^{r}g(X)(\mathfrak{C}_{2}+XL[X]).$ That
is $C_{1}^{\prime }=X(\mathfrak{C}_{1}+XL[X])$ and $C_{2}^{\prime }=X(%
\mathfrak{C}_{2}+XL[X])$ is a condensed pair. That is, for $c=X^{2}(\gamma
+Xh(X))$ we must find $c_{1}=X(\gamma _{1}+Xf(X))$ and $c_{2}=X(\gamma
_{2}+Xg(X))$ to get $c_{1}c_{2}=X^{2}(\gamma _{1}+Xf(X))(\gamma
_{2}+Xg(X))=X^{2}(\gamma _{1}\gamma _{2}+\gamma _{1}Xg(X)+\gamma
_{2}Xf(X)+X^{2}f(X)g(X))=X^{2}(\gamma +Xh(X))=c.$ Comparing the coefficients
of $X^{2}$ we must have $\gamma =\gamma _{1}\gamma _{2}$ where $\gamma
_{i}\in \mathfrak{C}_{i}$, as desired. (This leaves the case of $Xh(X)$ not
being a product, as indicated. The situation can be resolved by taking $%
c_{1}=X\gamma _{1}$ and $c_{2}=X(\gamma _{2}+(X/\gamma _{1})h(X)).$ For the
moreover part, observe that as we have noted $[L:K]\geq 4$ implies that $%
K\subseteq L$ is not $vs$-closed.
\end{proof}

Now one can go mimicking the $vs$-closed idea of \cite{ADu} by letting $M,N$ 
$D$-submodules of $L$ and letting $P(M,N)=\{mn|m\in M,n\in N\}$, letting $MN$
be the module generated by $P(M,N),$ and calling $D\subseteq L$ $sm$-closed
(submodule closed), if for every pair of two generated submodules $M,N$ one
has $P(M,N)=MN.$ Repeating the steps taken in the proofs of Theorem \ref%
{Theorem F} and Proposition \ref{Proposition J} one can prove the following
theorem.

\begin{corollary}
\label{Corollary L} Let $D$ be a domain, $K$ the quotient field of $D$, let $%
L$ be an extension of $L$ and let $X$ be an indeterminate over $L.$ Then the
following hold. (1) $D+XL[X]$ is condensed if and only if $D$ is condensed
and $D\subseteq L$ is $sm$-closed. (2) If $D+XL[X]$ is condensed, $[L:K]\leq
3.$
\end{corollary}

\begin{proof}
We leave (1) for an interested reader and for (2) we note that if $R=D+XL[X]$
is condensed and if $S=D\backslash \{0\},$ then so is $R_{S}=K+XL[X]$ and
this forces $[K:L]\leq 3.$
\end{proof}

This study may give us a number of examples and indirect results such as the
following. The go to reference for the following examples is \cite{AAZ}.

\begin{example}
\label{Example M} (1) Let $K\subseteq L$ be an extension of fields with $K=Q$
the field of rational numbers and $L$ a quadratic extension of $Q.$ Then $%
K+XL[X]$ is atomic, and condensed and hence cannot be a $\ast $-domain, nor
a pre-Schreier domain.
\end{example}

(2) With $Q$ and $L$ as above, $Q+XL[X]$ is atomic, and condensed, with the
property that every overring is atomic. This is because the integral closure
of $Q+XL[X]$ is $L[X]$ \cite{DSZ}. Of course if $[L:K]<\infty ,$ every
overring of $K+XL[X]$ would still be atomic, but in most cases the ring is
not condensed.

(3) Let $K\subseteq L$ be an extension of fields with $D=K+XL[X]$ condensed.
Then the following are equivalent. (a) $D$ is a $\ast $-domain, (b) $D$ is a
PID, (c) $D$ is pre-Schreier, (d) $D$ is integrally closed and (e) $K=L.$
(a) $\Rightarrow $ (b) (A condensed star domain is pre-Schreier), a
pre-Schreier atomic domain is a UFD and a condense UFD is a PID, (b) $%
\Rightarrow $ (c) Obvious (c) $\Rightarrow $ (a) a pre-Schreier domain is a $%
\ast $-domain, (b) $\Rightarrow $ (d) a PID is integrally closed (d) $%
\Rightarrow $ (b) An integrally closed condensed domain is Bezout and an
atomic Bezout domain is a PID. Finally, the equivalence of (d) and (e) is
obvious.

Theorem \ref{Theorem F} can be used to prove that if $D$ is condesed and if $%
K$ is a quotient field of $D,$ then $D+XK[X]$ is condensed.

\end{document}